\documentclass{proc-l}

\usepackage{xcolor}

\newtheorem{thm}{Theorem}[section]
\newtheorem{lem}[thm]{Lemma}
\newtheorem{cor}[thm]{Corollary}
\newtheorem{pro}[thm]{Proposition}
\theoremstyle{definition}

\theoremstyle{remark}
\newtheorem{remark}[thm]{Remark}

\numberwithin{equation}{section}

\begin{document}

\title{Generic properties of discrete Steklov eigenfunctions}



\author{Lixiang Chen}
\address{School of Mathematical Sciences, Harbin Engineering University, Harbin, Heilongjiang, PR China}
\curraddr{}
\email{clxmath@163.com}
\thanks{}

\author{BoBo Hua}
\address{School of Mathematical Sciences, LMNS, Fudan University, Shanghai, PR China}
\curraddr{}
\email{bobohua@fudan.edu.cn}
\thanks{}

\author{Yongtang Shi}
\address{Center for Combinatorics, LPMC, Nankai University, Tianjin, PR China}
\curraddr{}
\email{shi@nankai.edu.cn}

\date{}

\dedicatory{}


\begin{abstract}
Let  $G=(V,E)$ be a finite connected  graph with boundary $B $. 
We prove that for a generic positive edge weight function $w \in \mathbb{R}^{|E|}$, the Steklov eigenvalues  of $(G,B,w)$ are simple and every Steklov eigenfunction does not vanish on the boundary.  More precisely, the exceptional weights are contained in  a zero set of a non-identically zero polynomial and hence form a set of Lebesgue measure zero and Hausdorff dimension at most $|E|-1$.
Our results provide a discrete extension of the genericity theorem for the Steklov problem on compact manifolds.
\end{abstract}

\maketitle
\section{Introduction}
The Steklov eigenvalue problem, originating from Steklov's century-old work on liquid sloshing \cite{Kuznetsov2014,Stekloff1902}, is an active field connecting spectral theory, geometry, and mathematical physics. 
Let $(M,g)$ be a compact smooth Riemannian manifold of dimension $d\geq 2$ with smooth boundary $\partial M$.
The Riemannian metric $g$ induces the Dirichlet-to-Neumann operator $\Lambda_g:C^\infty(\partial M)\rightarrow C^\infty(\partial M)$, which maps a function on $\partial M$ to the normal derivative of its harmonic extension. The eigenvalues and eigenfunctions of $\Lambda_g$ are called the Steklov eigenvalues and Steklov eigenfunctions, respectively.

The geometric study of Steklov eigenvalues was initially motivated by Weinstock's 1954 result \cite{Weinstock1954}, which states that among simply connected planar domains with fixed perimeter, the first nontrivial Steklov eigenvalue is maximized uniquely by the disk. Later, Fraser and Schoen studied extremal Steklov eigenvalue problems for surfaces, higher-dimensional domains, and higher Steklov eigenvalues \cite{FraserSchoen2011,FraserSchoen2013,FraserSchoen2016,FraserSchoen2019,FraserSchoen2020}. 
In particular, they revealed a deep connection between extremal Steklov metrics and free boundary minimal surfaces in the Euclidean unit ball \cite{FraserSchoen2016}.  For a detailed account of recent advances on the Steklov eigenvalue problem, see the survey  \cite{ColboisGirouardGordonSher2023}.

In spectral theory, a genericity problem asks whether a given spectral property holds for a typical choice of the underlying geometric structure, usually after excluding a small exceptional set of metrics. Uhlenbeck's seminal work \cite{Uhlenbeck1976} showed that, for a generic Riemannian metric $g$, the nonzero eigenvalues of second-order elliptic operators, in particular those of the Laplace operator $\Delta_g$, are simple, and that the corresponding eigenfunctions satisfy certain non-degeneracy properties. This result has since been generalized in several directions, including to the Hodge Laplacian on closed 3-manifolds \cite{EncisoPeraltaSalas2012}, Dirac operators on spin 3-manifolds \cite{Dahl2003}, conformally covariant operators \cite{Canzani2014}, the Laplacian on metric graphs \cite{Friedlander2005}, and the Steklov operator on compact manifolds \cite{Wang2022}.

In the discrete setting, genericity results have also been established for the graph Laplacian. It was proved in \cite{Poignard2018} that, for a generic choice of edge weights, the Laplacian eigenvalues are simple and the corresponding eigenvectors have no zero entries. For the graph Laplacian, the edge weight function enters directly into the definition of the operator, and hence the Laplacian is naturally parameterized by the edge weights.  The Steklov operator is induced by the Dirichlet problem on a graph with boundary. Consequently, the effect of edge weights on the Steklov operator is transmitted through the interior vertices to the boundary, leading to a different dependence on the weight parameters. This distinction prevents the existing genericity methods for the Laplacian from being directly applied to the Steklov operator.

Motivated by these results, we study the corresponding genericity problem for the Steklov operator in the discrete setting.  More precisely, we consider the Steklov problem on weighted graphs with boundary.
All graphs in this paper are finite, simple, and undirected graphs with a non-empty boundary. 
Let $G=(V,E)$ be a finite connected graph with boundary vertex set $B\subseteq V$.
We use $\mathbb{R}_{+}$ to denote the set of positive real numbers.
A property $\Re$ of $(G,B)$ is said to hold for generic weights if the set of edge weights
\[
\{\, w\in \mathbb{R}_{+}^{|E|} : (G,B,w)\text{ does not satisfy }\Re \,\}
\]
has Lebesgue measure zero in $\mathbb{R}_{+}^{|E|}$. In this paper, we prove that the Steklov eigenvalues  of $(G,B,w)$ are simple for 
generic weights. Moreover, we prove that the corresponding Steklov eigenfunctions generically are nowhere zero on $B$.

Our main result can be stated as follows:
\begin{thm}[Main Theorem]\label{thm:generic}
Let $G=(V,E)$ be a finite connected  graph with boundary $B$. Then the following properties hold for generic weights:
\begin{enumerate} 
    \item  The eigenspaces of Steklov eigenvalues are one-dimensional, i.e.,  all Steklov eigenvalues of $(G,B)$ are simple;
    \item  All Steklov eigenfunctions of $(G,B)$ do not vanish on $B$.
\end{enumerate}
\end{thm}

\begin{remark}
\begin{enumerate}
    \item[(i)] The second assertion of Theorem~\ref{thm:generic} implies the first one, since if an eigenspace has dimension greater than one, then a suitable linear combination of eigenfunctions in this eigenspace would vanish at a prescribed boundary vertex.
    \item[(ii)] In fact, we prove that the set of exceptional weights, which do not satisfy the above assertions, is contained in a zero set of a non-identically zero  polynomial, hence is of Hausdorff dimension at most $|E|-1$ \cite[Theorem 1]{MURDZA2025103681}.
\end{enumerate}
\end{remark}

Since the set of generic weights is dense in the space of all weights $\mathbb{R}_{+}^{|E|}$, the properties stated in Theorem \ref{thm:generic} can be achieved by an arbitrarily small perturbation of a given weight. Consequently, even if the Steklov spectrum is not simple for a particular weight, or if some Steklov eigenfunction vanishes on the boundary, these degeneracies can be removed by a sufficiently small perturbation of the weights.

\begin{cor}\label{cor:mian}
Let $G=(V,E)$ be a finite connected graph with boundary $B$. For any weight $w \in \mathbb{R}_{+}^{|E|}$ and any $\varepsilon>0$, there exists a weight $\widetilde w \in \mathbb{R}_{+}^{|E|}$ with
$|\widetilde w(e)-w(e)|<\varepsilon$ for any $e \in E$ such that all Steklov eigenvalues of $(G,B,\widetilde w)$ are simple, and every Steklov eigenfunction is nonvanishing on $B$.
\end{cor}

\section{Preliminaries}

\subsection{The discrete Steklov operator}
Let $G=(V,E)$ be a finite connected  weighted graph with  weight function $w: E \rightarrow \mathbb{R}_{+}$, $\{x,y\} \mapsto w(x,y)$,   and let $B$ be the set of boundary vertices of $G$ with $|B| \geq 2$.
Denote by $\mathbb{R}^{V}$ the real function space defined on $V$, which is the $|V|$-dimensional Euclidean space.
For $u\in \mathbb{R}^V$, the \emph{Laplacian} $\Delta:\mathbb{R}^V \to \mathbb{R}^V$ is defined by
\begin{align*}
    \Delta u:~&V \to \mathbb{R}\\
    &x \mapsto (\Delta u)(x)=\sum_{\{x,y\}\in E}w(x,y)\left(u(x)-u(y)\right),
\end{align*}
and the discrete outward normal derivative $\partial_{\nu}:\mathbb{R}^V \to \mathbb{R}^{B}$ is defined by
\begin{align*}
   \partial_{\nu}u:~& B \to \mathbb{R}\\
    &x \mapsto (\partial_{\nu}u)(x)=\sum_{\{x,y\}\in E}w(x,y)\left(u(x)-u(y)\right).
\end{align*}
We use $B^c$ to denote the complement set of $B$ in $V.$ The Steklov eigenvalue problem on $(G,B,w)$ aims to find real numbers $\lambda$ and nontrivial functions $u \in \mathbb{R}^V$ such that
\begin{align*}
    \begin{cases}
       (\Delta u) (x)=0,  &\mbox{if $x\in B^c $},\\
        (\partial_{\nu}u)(x)=\lambda u(x), & \mbox{if $x\in B $}.
    \end{cases}
\end{align*}
Such a $\lambda$ is called a \emph{Steklov eigenvalue} of $(G,B,w)$. 
On the other hand, Steklov eigenvalues are the eigenvalues of the so-called \emph{Steklov operator} or Dirichlet-to-Neumann map:
\begin{align*}
 \Lambda: \mathbb{R}^{B} \rightarrow \mathbb{R}^{B}, f \mapsto \partial_{\nu}h_f.
\end{align*}
The eigenfunction $f$ of $\Lambda$ is called  a \emph{Steklov eigenfunction} of $(G,B,w)$, and $h_f \in \mathbb{R}^{V}$ is called the \emph{harmonic extension} of $f$.

To facilitate the use of matrix  techniques, we introduce the matrix forms of the Laplacian and the discrete Steklov operator, starting with some basic matrix terminology.
Let $M$ be an $n\times n$ matrix, and let $\alpha$ and $\beta$ be given index sets, i.e., subsets of $\{1,2,\ldots,n\}$.
We use $M_{[\alpha,\beta]}$ to  denote the submatrix of $M$ with rows indexed by $\alpha$ and columns indexed by $\beta$, both of which are thought of as increasingly ordered sequences, so the rows and columns of the submatrix appear in their natural order. We often write $M_{[\alpha]}$ for $M_{[\alpha,\alpha]}$.
The \emph{adjugate} $\operatorname{adj} M$ of $M$ is the transpose of the cofactor matrix of $M$, whose elements are given by $(\operatorname{adj} M)_{ij} = (-1)^{i+j} \det \bigl( M[ \{j\}^c, \{i\}^c ] \bigr)$.

We use  $L$ to denote the matrix representation of the Laplacian $\Delta$ of $G$.
With respect to the decomposition $V=B^c\sqcup B$ into interior and boundary vertices, $L$ admits the block form
\[
L=
\begin{pmatrix}
L_{[B^c]} & L_{[B^c,B]} \\
L_{[B,B^c]} & L_{[B]}
\end{pmatrix}.
\]
By the Matrix-Tree Theorem \cite[Theorem 7.2.2]{cve2010spectra}, every proper principal submatrix of the Laplacian matrix of a connected graph is nonsingular, hence $L_{[B^c]}$ is invertible.
The matrix representation $S$ of the Steklov operator $\Lambda$ on $G$ is the Schur complement of $L_{[B^c]}$ in $L$, namely
\begin{align}
 S&=L_{[B]}-L_{[B,B^c]}L^{-1}_{[B^c]}L_{[B^c, B]}  \label{eq:Shur} \\
  &=L_{[B]}-\frac{L_{[B,B^c]}\mathrm{adj}(L_{[B^c]})L_{[B^c, B]}}{\det(L_{[B^c]})}. \label{eq:adj}
\end{align}
In particular, since the entries of $\mathrm{adj}(L_{[B^c]})$ and $\det(L_{[B^c]})$ are polynomial functions of the entries of $L_{[B^c]}$, \eqref{eq:adj} shows that each entry of $S$ is a rational function of the entries of $L,$ and hence a rational function of the edge weights. The rational expression \eqref{eq:adj} is well-defined precisely when
$\det(L_{[B^c]})\neq 0$, which holds here because $G$ is connected. For $u,v\in B^c$, the entry $\mathcal{G}_B(u,v)=\bigl(L^{-1}_{[B^c]}\bigr)_{uv}$
is referred to as the \emph{discrete Dirichlet Green function} for the graph, as introduced by Chung and Yau \cite{Chung2000}.
For $x \in B^c$ and $y \in B$, define
\begin{align*}
  \mathcal{P}_B(x,y) =\sum_{z \in B^c, \{y,z\} \in E}\mathcal{G}_B(x,z)w(z,y).
\end{align*}
We call $\mathcal{P}_B(x,y)$ the \emph{discrete Poisson kernel}. Indeed, for a function $f$ on $B$ and  its harmonic extension $h_f$, the Poisson kernel gives the representation
\begin{align*}
  h_f(x) =\sum_{y \in B} \mathcal{P}_B(x,y)f(y).
\end{align*}
Therefore, the Steklov operator admits the following expression:
\begin{align}
\Lambda f(x) &=\sum_{\{x,y\} \in E}w(x,y)(h_f(x)-h_f(y)) \notag  \\
&=\sum_{\substack{y\in B\\ \{x,y\} \in E}}w(x,y)(f(x)-f(y))+\sum_{\substack{z\in B^c\\ \{x,z\} \in E}}w(x,z)\left(
f(x)-\sum_{y\in B}\mathcal{P}_B(z,y)f(y)
\right). \label{eq:Lambda}
\end{align}

Let $(G,w)$ be a weighted graph with underlying graph $G=(V,E)$, 
viewed as an electrical network by assigning to each edge $\{x,y\}\in E$ a resistor with resistance $1/w(x,y)$. 
The matrix $S$ of the Steklov operator is the Laplacian matrix of a weighted graph $(G_B,w_B)$ on the boundary vertex set $B$, where $G_B=(B,E_B)$.
This graph is the effective electrical network obtained from $G$ by eliminating the interior vertices in $B^c$ (also known as the Kron reduction of $G$).
Consequently, for any $x,y\in B$, the effective resistance between $x$ and $y$ in $(G_B,w_B)$ coincides with that in the original graph $G$. By \eqref{eq:Lambda} (or \eqref{eq:Shur}), we have
\begin{align*}
  \Lambda f(x) =\sum_{\substack{x,y\in B\\ x \neq y}} w_B(x,y)(f(x)-f(y)),
\end{align*}
where 
\begin{equation}\label{eq:weight}
w_B(x,y)=w(x,y)+\sum_{z,t \in B^c}w(x,z)\mathcal{G}_B(z,t)w(t,y).
\end{equation}

\begin{remark}
The connectedness assumption is not essential. If $G$ is disconnected, then the problem decomposes over the connected components of $G$. Components containing no boundary vertices do not contribute to the Steklov operator, while each component meeting $B$ contributes its own Steklov operator. Thus, after ordering the boundary vertices componentwise, $\Lambda_{(G,B)}=\bigoplus \Lambda_{(G_i,B_i)}$,
where $G_i$ are the connected components of $G$ and $B_i=B\cap V(G_i)$. Hence, the general disconnected case reduces to the connected case.
\end{remark}

\subsection{Matrix tools}
We begin by recalling several matrix-theoretic results that will serve as fundamental tools in the paper. These include a perturbation theorem for eigenvalues (Theorem \ref{thm:perturbation}), a stability theorem for simple eigenvalues and their associated eigenvectors (Theorem \ref{thm:sensi}), and a criterion for zero entries of eigenvectors of Hermitian matrices based on the interlacing theorem (Proposition \ref{pro:zero}).

The following conclusion is a direct corollary of the well-known Bauer--Fike theorem on matrix perturbation.
\begin{thm}\cite[Corollary 6.3.4]{Horn_Johnson_2012}\label{thm:perturbation}
Let $M$ and $E$ be $n$-dimensional matrices and suppose that $M$ is normal. If $ \hat{\lambda} $ is an eigenvalue of $M + E$, then there is an eigenvalue $ \lambda $ of $ M $ such that $ |\hat{\lambda} - \lambda| \leq \|E\|_2 $.
\end{thm}

Theorem~\ref{thm:perturbation} shows that eigenvalues are stable under matrix perturbations. The next result shows that, for a simple eigenvalue, the corresponding eigenvector also depends stably on the entries of the matrix.

\begin{thm}\cite[Theorem 2.1]{Sun1990}\label{thm:sensi}
Let $\mathcal{F}$ be an open subset of $\mathbb{R}^{N}$, and let
$p=(p_1,\ldots,p_N)^{\top}\in\mathcal{F}$. Let $A(p)$ be a real symmetric
$n\times n$ matrix whose entries depend analytically on $p_1,\ldots,p_N$.
Suppose that $\lambda$ is a simple eigenvalue of $A(p^*)$ for some
$p^*\in\mathcal{F}$, and let $\mathbf{x}$ be a corresponding eigenvector.
Then there exist a neighborhood $\mathcal{B}(p^*)\subseteq\mathcal{F}$ of
$p^*$, a real analytic function $\lambda(p)\in\mathbb{R}$, and a real
analytic function $\mathbf{x}(p)\in\mathbb{R}^n$, such that
\[
A(p)\mathbf{x}(p)=\lambda(p)\mathbf{x}(p),\qquad p\in\mathcal{B}(p^*),
\]
and
\[
\lambda(p^*)=\lambda,\qquad \mathbf{x}(p^*)=\mathbf{x}.
\]
\end{thm}

\begin{thm}\cite[Cauchy's interlacing theorem]{Horn_Johnson_2012}\label{thm:interlacing}
Let $B\in M_n$ be Hermitian, let $y\in \mathbb{C}^n$ and $a\in \mathbb{R}$ be given, and let $A= \begin{bmatrix}
B & y\\
y^* & a
\end{bmatrix}
\in M_{n+1}
$, where $y^*$ denotes the conjugate transpose of $y$.
Then
\begin{align}\label{eq:cauchy}
\lambda_1(A)\le \lambda_1(B)\le \lambda_2(A)\le \cdots
\le \lambda_n(A)\le \lambda_n(B)\le \lambda_{n+1}(A).
\end{align}
Moreover, $\lambda_i(A)=\lambda_i(B)$ if and only if there is a nonzero $z\in \mathbb{C}^n$ such that $ Bz=\lambda_i(B)z$, $y^*z=0$,  $Bz=\lambda_i(A)z$;
and $\lambda_i(B)=\lambda_{i+1}(A)$ if and only if there is a nonzero $z\in \mathbb{C}^n$ such that $Bz=\lambda_i(B)z$,  $y^*z=0$,  $Bz=\lambda_{i+1}(A)z$.
If no eigenvector of $B$ is orthogonal to $y$, then every inequality in \eqref{eq:cauchy} is a strict inequality.
\end{thm}

The following proposition, based on Cauchy's interlacing theorem, gives a useful criterion for a Hermitian matrix to admit an eigenvector with a zero entry.
We emphasize that the Hermitian assumption cannot be dropped, since the sufficiency part may fail for non-Hermitian matrices.

\begin{pro}\label{pro:zero}
Let $A \in  M_n$ be a Hermitian matrix. Then $A$ admits an eigenvector having a zero entry  if and only if $A$ shares an eigenvalue with one of its principal submatrices of order $n-1$.
\end{pro}

\begin{proof}
Suppose that $A$ shares an eigenvalue $\lambda$ with a principal submatrix $B$ of order $n-1$. After a suitable permutation similarity, we may assume that $A$ is partitioned as in Theorem \ref{thm:interlacing}.  By the equality case in Cauchy's interlacing theorem, $A$ shares an eigenvalue $\lambda$ with $B$ if and only if there exists  an eigenvector  $z$ of $B$ corresponding to $\lambda$ such that $y^{*}z=0$.
One easily checks that
$\begin{bmatrix}
z\\
0
\end{bmatrix}
$
is an eigenvector of $A$ corresponding to $\lambda$. This  completes  the proof.
\end{proof}

\begin{thm}\cite[2.6.5]{Herbert_1996}\label{thm:map}
For every nonzero polynomial map $p:\mathbb{R}^n \rightarrow \mathbb{R}$, the $n$-dimensional Lebesgue measure of the set $\{x:p(x)=0\}$ is zero.
\end{thm}

\subsection{Some properties of discrete Steklov operators and eigenvalues}

A path in $G$ whose all internal vertices are not boundary vertices of $G$ is called an interior path. In particular, an edge $\{x,y\}$ for $x,y\in B$ is regarded as an interior path.
\begin{pro}\label{pro:interior}
Let $(G,B,w)$ be a finite connected  weighted graph with boundary $B$.
Let $G_B=(B,E_B)$, and let $(G_B,w_B)$ denote the effective electrical network associated with $(G,B,w)$ obtained by eliminating the interior vertices in $B^c$.
Then, for any two distinct vertices $x,y\in B$,  $\{x,y\}\in E_B$ if and only if there exists an interior path between $x$ and $y$ in $G$.
\end{pro}

\begin{proof}
Let $C_1,\ldots,C_m$ be the connected components of the subgraph of $G$ induced by $B^c$. Without loss of generality, we may write
\[
L_{[B^c]}
=
\operatorname{diag}(L_1,\ldots,L_m),
\]
where $L_i$ is induced by the connected component $C_i$.
Thus, for each $i=1,\ldots,m$, the matrix $L_i$ is an irreducible $M$-matrix. Note that each $L_i$ is a nonsingular irreducible $M$-matrix.
It is known that the inverse of a nonsingular irreducible $M$-matrix is positive \cite[7.2.P31]{Horn_Johnson_2012}.
Thus, Green function $\mathcal{G}(u,v)>0$ if and only if $u$ and $v$ belong to the same connected component $C_i$; otherwise $\mathcal{G}(u,v)=0$.
Equivalently, there exists an interior path between $u$ and $v$ whose vertices all belong to $B^c$.

For distinct vertices $x,y\in B$, by \eqref{eq:weight}, we have
\begin{align*}
w_B(x,y)=w(x,y)+\sum_{u,v \in B^c}w(x,u)\mathcal{G}(u,v)w(v,y).
\end{align*}
We conclude that $w_B(x,y) >0$ if and only if $w(x,y) >0$  or there exist vertices $u,v\in B^c$ such that $\{x,u\}\in E$, $\{v,y\}\in E$ and $\mathcal{G}(u,v)>0$. This completes the proof. 
\end{proof}

\begin{remark}
Hua et al.~\cite[Theorem 2.1]{Hua2025} give an expression of $w_B(x,y)$ in terms of the effective capacity of an electrical network. Together with standard electrical network theory, Proposition \ref{pro:interior} also follows from their result. 
We provide a proof here for completeness.
\end{remark}

\begin{pro}\label{pro:contracting}
Let $G$ be a finite connected weighted graph with boundary $B$.
Suppose that an interior vertex $v\in B^c$ has exactly two neighbours $u_1,u_2,$ and $\{u_1,u_2 \} \notin E.$
Let $(G',B,w')$ be obtained from $(G,B,w)$ by contracting the path $u_1-v-u_2$ to a single edge $\{u_1,u_2\}$ with weight $w'(u_1,u_2)=\frac{w(u_1,v)w(v,u_2)}{w(u_1,v)+w(v,u_2)}$.
Then $(G,B,w)$ and $(G',B,w')$ have the same Steklov matrix.
\end{pro}

\begin{proof}
By the series law for resistors, the path $u_1-v-u_2$ is equivalent to a single edge ${u_1,u_2}$ with conductance $ \left(\frac{1}{w(u_1,v)}+\frac{1}{w(v,u_2)}\right)^{-1}=w'(u_1,u_2)$.
Thus, the effective electrical network $G_B$ is unchanged, and hence the Steklov matrix is the same.
\end{proof}

\begin{lem}\cite[Lemma A.2]{Poignard2018}\label{lem:Laplacepath}
For path $P=(V_P,E_P)$, there exists a positive weight function $w_P\in \mathbb{R}^{|E_P|}_{+}$ such that all Laplacian eigenvalues of the weighted graph $(P,w_P)$ are simple and every Laplacian eigenfunction of $(P,w_P)$ is nowhere zero on  $V_P$.
\end{lem}

\begin{cor}\label{cor:path}
Let $P=(V_P,E_P)$ be a path with boundary $B_P$, whose endpoints lie in $B_P$. There exists a positive edge-weight function $w_P\in\mathbb{R}_{+}^{|E_P|}$ such that every Steklov eigenvalue of $(P,B_P,w_P)$ is simple, and the harmonic extension of every Steklov eigenfunction does not vanish on $V_P$.
\end{cor}
\begin{proof}
Let $B_P=\{b_1,\ldots, b_m\}$ be the boundary of $P$, ordered along with the path.
Let $Q=(B_P,E_Q)$ be the path on $B_P$ with $E_Q=\{\{b_i,b_{i+1}\}:1\le i\le m-1\}$.
By Lemma \ref{lem:Laplacepath}, there exists a positive weight $\widetilde{w}$ on $Q$ such that  all Laplacian eigenvalue of the weighted graph $(Q,\widetilde{w})$ are simple and every Laplacian eigenfunction $f$ of $(Q,\widetilde{w})$ is nowhere zero on  $B_P$.
For each subpath $P_i$ between $b_i$ and $b_{i+1}$, write $P_i : x_0=b_i,x_1,\ldots,x_k=b_{i+1}$.
We use $r_j$ to denote the resistance between $x_{j-1}$ and $x_{j}$, i.e., $r_j=\frac{1}{w_P(x_{j-1},x_j)}$.
Choose positive edge weight function $w_P$ so that
\begin{align}\label{eq:condition}
\left(\sum_{j=1}^{k}r_j\right)^{-1}=\widetilde{w}(b_i,b_{i+1}).
\end{align}
By Proposition \ref{pro:contracting}, the Laplacian eigenvalues and eigenfunctions of $(Q,\widetilde{w})$ are exactly the Steklov eigenvalues and eigenfunctions of $(P,B_P,w_P)$.
It remains to ensure that the harmonic extensions $h$ of Steklov  eigenfunctions  $f$ have no zeros at the interior vertices of $P$.

Since $h$ is harmonic on $P_i$ with $h(x_0)=f(b_i)$ and $h(x_k)=f(b_{i+1})$, then for each $j$,
\begin{align*}
  h(x_j)=(1-\theta_j)f(b_i)+\theta_j f(b_{i+1}),\qquad 0 \leq \theta_j \leq 1.
\end{align*}
Here, $\theta_j$ depends only on the weights of $P_i$, in fact, $\theta_j=\frac{\sum_{t=1}^{j}r_t}{\sum_{t=1}^{k}r_t}$.
Hence,  $h(x_j)=0$ for at most one value of $\theta_j$.
Since there are only finitely many Steklov eigenfunctions $f$ up to scalar multiples and finitely many interior vertices, there are only finitely many such values to avoid.
We can choose the weight function $w_P$ on each $P_i$, still satisfying \eqref{eq:condition}, so that the corresponding $\theta_j$'s avoid all these values. Then no harmonic extension vanishes at an interior vertex.
This completes the proof.
\end{proof}

An edge of a graph is called a \emph{pendant edge} if it is incident with a vertex of degree one. Such a vertex is called a \emph{pendant vertex} (or a leaf). More generally, a \emph{pendant path} is a path whose internal vertices have degree two and whose one end vertex is a pendant vertex while the other end vertex is attached to the remaining part of the graph. In particular, a pendant edge can be viewed as a pendant path of length one. The following lemma describes how the Steklov matrix of a weighted graph changes when a pendant edge is attached to the graph.
It is obtained by a direct computation with the Schur complement in \eqref{eq:Shur}.

\begin{lem}\label{lem:pendantedge}
Let $(G,B,w)$ be a connected  weighted graph with boundary $B$, and write its Laplacian matrix in block form as $ L= \begin{pmatrix} L_{[B^c]} & L_{[B^c,B]}\\[2pt] L_{[B,B^c]} & L_{[B]} \end{pmatrix}$.
Let $(\widehat{G}, B\cup\{b\}, \widehat{w})$ be the weighted graph obtained from $G$ by attaching a pendant edge to a vertex $u\in V(G)$, where the new pendant vertex is $b$ and the weight of the new edge is $\tau$. Let $S$ and $\widehat{S}$ denote the Steklov matrices of $(G,B,w)$ and $(\widehat{G}, B\cup\{b\}, \widehat{w})$, respectively.
Then
\[ \widehat{S}= \begin{cases} \displaystyle \begin{pmatrix} S & \mathbf0\\ \mathbf 0^{\top} & 0 \end{pmatrix} +\tau \begin{pmatrix} \mathbf e_u\mathbf e_u^{\top} & -\mathbf e_u\\ -\mathbf e_u^{\top} & 1 \end{pmatrix}, & u\in B, \\[16pt] \displaystyle \begin{pmatrix} S & \mathbf 0\\ \mathbf 0^{\top} & 0 \end{pmatrix} +\frac{\tau}{1+\tau\alpha} \begin{pmatrix} \mathbf{p}_u\mathbf{p}_u^{\top} & -\mathbf{p}_u\\ -\mathbf{p}_u^{\top} & 1 \end{pmatrix}, & u\notin B, \end{cases} \]
where $\mathbf e_u$ is the standard basis vector indexed by the vertex $u$, of the appropriate dimension, and $\mathbf{p}_u=-L_{[B,B^c]}L_{[B^c]}^{-1}\mathbf e_u$ and where $ \alpha=\mathbf e_u^{\top}L_{[B^c]}^{-1}\mathbf e_u$.
\end{lem}
\begin{proof}
When $u \in B$,  the Laplacian matrix $\widehat{L}$ of $(\widehat{G}, B\cup\{b\}, \widehat{w})$ takes the block form
\[
\widehat{L} = \begin{pmatrix}
L_{[B^c]} & L_{[B^c,B]} & \mathbf{0} \\[2pt]
L_{[B,B^c]} & L_{[B]} + \tau\mathbf{e}_u\mathbf{e}_u^{\top} & -\tau\mathbf{e}_u \\[2pt]
\mathbf{0}^{\top} & -\tau\mathbf{e}_u^{\top} & \tau
\end{pmatrix}.
\]
By the Schur complement formula \eqref{eq:Shur}, we have
\[
\widehat{S}= \begin{pmatrix}
L_{[B]}+\tau\mathbf{e}_u\mathbf{e}_u^{\top} & -\tau\mathbf{e}_u \\[2pt]
-\tau\mathbf{e}_u^{\top} & \tau
\end{pmatrix}
- \begin{pmatrix}
L_{[B,B^c]} \\[2pt] \mathbf{0}^{\top}
\end{pmatrix} L_{[B^c]}^{-1}
\begin{pmatrix}
L_{[B^c,B]} & \mathbf{0}
\end{pmatrix}.
\]
Note that the second term equals $\begin{pmatrix} L_{[B,B^c]}L_{[B^c]}^{-1}L_{[B^c,B]} & \mathbf{0} \\[2pt] \mathbf{0}^{\top} & 0 \end{pmatrix}$. Hence, we have
\begin{align*}
  \widehat{S} = \begin{pmatrix}
S  & 0 \\[4pt]
0 & 0
\end{pmatrix}+\tau\begin{pmatrix}
 \mathbf{e}_u\mathbf{e}_u^{\top} & -\mathbf{e}_u \\[4pt]
-\mathbf{e}_u^{\top} & 1
\end{pmatrix}.
\end{align*}

When  $u \in B^c$, we have
\[
\widehat{L} = \begin{pmatrix}
L_{[B^c]}+\tau\mathbf{e}_u\mathbf{e}_u^{\top} & L_{[B^c,B]} & -\tau\mathbf{e}_u \\[2pt]
L_{[B,B^c]} & L_{[B]} & \mathbf{0} \\[2pt]
-\tau\mathbf{e}_u^{\top} & \mathbf{0}^{\top} & \tau
\end{pmatrix}.
\]
By \eqref{eq:Shur},  we have
\begin{align*}
\widehat{S} = \begin{pmatrix}
L_{[B]} & \mathbf{0} \\[2pt] \mathbf{0}^{\top} & \tau
\end{pmatrix}
- \begin{pmatrix}
L_{[B,B^c]} \\[2pt] -\tau\mathbf{e}_u^{\top}
\end{pmatrix}
\bigl(L_{[B^c]}+\tau\mathbf{e}_u\mathbf{e}_u^{\top}\bigr)^{-1}
\begin{pmatrix}
L_{[B^c,B]} & -\tau\mathbf{e}_u
\end{pmatrix}.
\end{align*}
Using the Sherman–Morrison–Woodbury identity \cite[(0.7.4.2)]{Horn_Johnson_2012}, we have
\[
\bigl(L_{[B^c]}+\tau\mathbf{e}_u\mathbf{e}_u^{\top}\bigr)^{-1} = L_{[B^c]}^{-1} - \frac{\tau}{1+\tau\alpha}\,L_{[B^c]}^{-1}\mathbf{e}_u\mathbf{e}_u^{\top}L_{[B^c]}^{-1},e
\]
where $\alpha=\mathbf{e}_u^{\top}L_{[B^c]}^{-1}\mathbf{e}_u > 0$.
Set $\mathbf{p}_u = -L_{[B,B^c]}L_{[B^c]}^{-1}\mathbf{e}_u$ (here, $\mathbf{e}_u \in \mathbb{R}^{|B^c|}$), then a direct computation yields
\begin{align*}
  \widehat{S} &= \begin{pmatrix}
L_{[B]} & \mathbf{0} \\[2pt] \mathbf{0}^{\top} & \tau
\end{pmatrix}
- \begin{pmatrix}
L_{[B,B^c]}L_{[B^c]}^{-1}L_{[B^c,B]} - \dfrac{\tau}{1+\tau\alpha}\mathbf{p}_u\mathbf{p}_u^{\top} & \dfrac{\tau}{1+\tau\alpha}\mathbf{p}_u \\[6pt]
\dfrac{\tau}{1+\tau\alpha}\mathbf{p}_u^{\top} & \dfrac{\tau^2\alpha}{1+\tau\alpha}
\end{pmatrix} \notag \\
&=\begin{pmatrix}
S & \mathbf{0}  \\[4pt]
\mathbf{0}^{\top} & 0
\end{pmatrix} + \dfrac{\tau}{1+\tau\alpha}\begin{pmatrix}
\mathbf{p}_u\mathbf{p}_u^{\top} & -\mathbf{p}_u \\[4pt]
-\mathbf{p}_u^{\top} & 1
\end{pmatrix}.
\end{align*}
\end{proof}

\begin{remark}
In the setting of Lemma~\ref{lem:pendantedge}, the vector $\mathbf e_u \in \mathbb{R}^{|B|}$ can be identified with the Kronecker delta function  at the  vertex $u \in B$, that is $<\mathbf e_u, f>=\sum_{x}\delta_{ux}f(x)=f(u)$  for a function $f$ on $B.$ 
For  $u\in B^c$, the quantity $\alpha=\mathbf e_u^{\top}L_{[B^c]}^{-1}\mathbf e_u$ (here, $\mathbf e_u \in \mathbb{R}^{|B^c|}$) is the value of the discrete Dirichlet Green function at $(u,u)$, that is, $\alpha=\mathcal{G}_{B}(u,u)$. Moreover, the vector $\mathbf p_u=-L_{[B,B^c]}L_{[B^c]}^{-1}\mathbf e_u$ can be interpreted as the discrete Poisson kernel $\mathcal{P}_B(u,\cdot)$ associated with the interior point $u$ and the boundary $B$. Consequently,  the harmonic extension $h_f(u)=<\mathbf p_u,f>=\sum_{x\in B}\mathcal{P}_B(u,x)f(x)$.
\end{remark}

\section{Proof of main results}\label{sec:main}

\subsection{Proof idea and sketches of the proofs}
We briefly explain the main ideas in the proof of Theorem \ref{thm:generic}. 
The proof consists of two parts, corresponding to the generic simplicity of Steklov eigenvalues and the generic nonvanishing of Steklov eigenfunctions on the boundary.

Both parts follow the same strategy: we first characterize the exceptional sets by the vanishing of discriminants, which are rational functions of the edge weights, and then show that these discriminants are not identically zero. The desired genericity then follows from the zero-set theorem (Theorem \ref{thm:map}).

For the first part concerning the simplicity of eigenvalues, it suffices to construct a choice of edge weights for $(G,B,w)$ such that all Steklov eigenvalues are simple. We reduce this construction to a spanning tree of $G$ by assigning zero weights to all non-tree edges. This choice  guarantees that the discriminant is well-defined due to the connectedness of the spanning tree. We then construct the edge weights on this spanning tree inductively.
Starting from a weighted path with simple Steklov spectrum (by Corollary \ref{cor:path}), we add the remaining boundary vertices one by one. At each step, several new edges are introduced simultaneously.  By the perturbation theorem (Theorem \ref{thm:perturbation}), we show  that when the weights of these newly added edges are chosen sufficiently small, the eigenvalues remain simple.

The second part concerning the nonvanishing property of eigenfunctions follows from the same inductive construction. We again start with the weighted path provided by Corollary~\ref{cor:path}, which has both a simple Steklov spectrum and nowhere-vanishing harmonic extensions of all Steklov eigenfunctions. During the inductive process of adding boundary vertices one at a time, several new edges are introduced at each step. For all eigenvalues except the first nontrivial one, the sensitivity theorem for simple eigenvalues (Theorem~\ref{thm:sensi}) implies that the corresponding eigenvectors remain nonzero at every boundary vertex.
The only additional difficulty concerns the eigenvector corresponding to the first nontrivial eigenvalue, which bifurcates from the eigenspace associated with a multiple zero eigenvalue of the limiting block matrix. By analyzing the limiting eigenspace, we show that this eigenvector is also nonzero at every boundary vertex.

This completes the inductive proof.
\subsection{Proofs}

\noindent{\textbf{Part 1 of Theorem \ref{thm:generic}}}
All Steklov eigenvalues of $(G,B)$ are simple for  generic weights.

\begin{proof}
Since the case $|B|=1$ is trivial, in what follows, we assume $|B| \geq 2$.
For a fixed weight function $w\in \mathbb{R}_{+}^{|E|}$, let $\chi(\lambda,w)$ denote the characteristic polynomial of the Steklov matrix $S$ associated with $(G,B,w)$. For each $i\in B$, let $\chi_i(\lambda,w)$ denote the characteristic polynomial of the principal submatrix $S_{[B\setminus\{i\},\,B\setminus\{i\}]}$. Note that the determinant of a matrix is a homogeneous polynomial in its entries. It follows from \eqref{eq:adj} that the entries of $S$ are rational functions in the weight variables $w$.
Hence, the coefficients of $\chi(\lambda,w)$ and $\chi_i(\lambda,w)$ belong to the field of rational functions in the variables $w$.



Given two real polynomials $f,g\in \mathbb{R}[x]$ of positive degree, the resultant of $f$ and $g$, denoted by $\operatorname{Res}(f,g)$, is an integer polynomial in the indeterminates given by the coefficients of $f$ and $g$. Moreover,  $\operatorname{Res}(f,g)=0$ if and only if $f$ and $g$ have a common root in $\mathbb{C}$; see \cite[Page.~78]{usingalgebraic}. For a real polynomial $f$,  
$\operatorname{Res}(f,f')=0,$ if and only if $f$ has a multiple root.

Motivated by this property, we consider the following map:
\begin{align*}
\operatorname{Discr}_s:\mathbb{R}_{+}^{|E|}
&\to \mathbb{R}, \\
w
&\mapsto
\operatorname{Res}\left(\chi(\lambda,w),\sum_{i\in B}\chi_i(\lambda,w)\right).
\end{align*}
It is known that $\sum_{i\in B}\chi_i(\lambda,w)=\chi'(\lambda,w)$, where $\chi'(\lambda,w)$ denotes the derivative of $\chi(\lambda,w)$ with respect to $\lambda$ \cite[Eq.~(0.8.10.2)]{Horn_Johnson_2012}.
Hence, $\operatorname{Discr}_s(w)=0$ if and only if $(G,B,w)$ has a multiple Steklov eigenvalue.

Recall that $\operatorname{Discr}_s(w)$ is a rational function of the weight variables $w$. More precisely, it can be written as
\[
\operatorname{Discr}_s(w)=\frac{P_s(w)}{\det(L_w[B^c])^N}
\]
for some nonnegative integer $N$, where $P_s(w)$ is a polynomial. Since $G$ is connected and the edge weights are positive, we have
$\det(L_w[B^c])>0$ on $\mathbb{R}^{|E|}_+$. Hence, $\operatorname{Discr}_s(w)=0$ if and only if $P_s(w)=0$.
By Theorem~\ref{thm:map}, if  $\operatorname{Discr}_s$  is not identically zero on $\mathbb{R}_{+}^{|E|}$, then  $\operatorname{Discr}_s^{-1}(0)$ has Lebesgue measure zero in $\mathbb{R}_{+}^{|E|}$.
Therefore, to complete the proof, it remains to show that $\operatorname{Discr}_s$ is not identically zero.

Let $T=(V_T,E_T)$ be a spanning tree of $G$. We first choose a positive weight function $w_T\in\mathbb{R}_{+}^{|E_T|}$ on $T$, and then extend it to a nonnegative weight function $w^*\in\mathbb{R}^{|E|}$ on $G$ by setting
$$
w^*(e)=
\begin{cases}
w_T(e), & e\in E_T,\\
0, & e\in E\setminus E_T.
\end{cases}
$$
The rational function $\operatorname{Discr}_s$ is well-defined at $w^*$ due to the connectivity of $T$.
The weighted graph $(G,B,w^*)$ reduces to the weighted tree $(T,B,w_T)$, and then
\begin{align*}
 \operatorname{Discr}_s(w^*) &=\operatorname{Discr}_s^{T}(w_T)
\end{align*}
where the right-hand side denotes the corresponding discriminants for the weighted tree $(T,B,w_T)$. Therefore, it  suffices to find a positive edge weight function $w_T$ on $T$ for which the discriminant is nonzero.
Without loss of generality, we may assume that all leaves of $T$ lie in $B$, since removing a leaf not in $B$ does not change the Steklov operator.

Choose a nontrivial weighted path $(P,B_P,w_P)$ in $(T,B,w_T)$ such that its end-vertices are in $B$.
By Corollary \ref{cor:path}, we may choose the weights $w_P$ such that all Steklov eigenvalues of $(P,B_P,w_P)$ are simple.

Let $b\in B\setminus B_P$ be a boundary vertex that is connected by  an interior path in $T$  to some vertex in $B_P$.
Let $(T^{(1)},B^{(1)},w^{(1)})$ be the minimal weighted  subtree of $T$ containing $B_P\cup\{b\}$, so that $B^{(1)}=B_P\cup\{b\}$.
Observe that $T^{(1)}$ is obtained from $P$ by attaching a pendant path $P_{ub}$ from some vertex $u\in V(P)$ to $b$. Let  $(\widehat{T},B^{(1)},\widehat{w})$ be the  weighted tree obtained via contracting interior vertices on $P_{ub}$ of $(T^{(1)},B^{(1)},w^{(1)}),$ replacing $P_{ub}$ by a single edge $\{u,b\}$ with weight $\widehat{w}(u,b) =\left(\sum_{e \in E(P_{ub})}\frac{1}{w^{(1)}(e)}\right)^{-1}.$ Now we write $\tau:=\widehat{w}(u,b),$ regarded as the parameter in the following. Let $\widehat{S}=\widehat{S}(\tau)$ be the Steklov matrix of 
the  weighted tree $(\widehat{T},B^{(1)},\widehat{w}),$ which coincides with the Steklov matrix of $T^{(1)}$ by Proposition~\ref{pro:contracting}.
From Lemma \ref{lem:pendantedge}, we have
\begin{align}\label{eq:S}
  \widehat{S}= \begin{cases} \displaystyle \begin{pmatrix}
S_P  & 0 \\[4pt]
0 & 0
\end{pmatrix}+\tau\begin{pmatrix}
 \mathbf{e}_u\mathbf{e}_u^{\top} & -\mathbf{e}_u \\[4pt]
-\mathbf{e}_u^{\top} & 1
\end{pmatrix}, & u\in B, \\[16pt] \displaystyle \begin{pmatrix} S_P & \mathbf 0\\ \mathbf 0^{\top} & 0 \end{pmatrix} +\frac{\tau}{1+\tau\alpha} \begin{pmatrix} \mathbf{p}_u\mathbf{p}_u^{\top} & -\mathbf{p}_u\\ -\mathbf{p}_u^{\top} & 1 \end{pmatrix}, & u \in  B^c. \end{cases}
\end{align}

For $\tau>0,$ we order the eigenvalues of $\widehat{S}(\tau)$  as
\begin{equation}\label{eq:eigenorder}
    0=\hat{\lambda}_0(\tau) < \hat{\lambda}_1(\tau)  \leq \hat{\lambda}_2(\tau) \leq \cdots \leq \hat{\lambda}_{|B_P|}(\tau),
\end{equation} where $\hat{\lambda}_1(\tau)>0$ follows from the connectivity of the graph. By the perturbation theorem (Theorem~\ref{thm:perturbation}) (see also Kato's general perturbation theory~\cite[Theorems II.5.1 and II.5.2]{kato1995perturbation}), since
\[
\widehat{S}(\tau)\to
\begin{pmatrix}
S_P & \mathbf{0}\\[4pt]
\mathbf{0}^{\top}&0
\end{pmatrix}
=:S_0,
\qquad \tau\to0,
\]
the eigenvalue $\hat{\lambda}_i(\tau)$ extends to a continuous function on $[0,\infty),$ and it converges to the corresponding eigenvalues of the limit matrix $S_0$ as $\tau\to 0$.
Note that the eigenvalues of $S_0$ are 
$$0,\lambda_1(S_P)=0,\lambda_2(S_P),\cdots,\lambda_{|B_P|}(S_P).$$ Since $0<\lambda_2(S_P)<\cdots<\lambda_{|B_P|}(S_P),$ there exists $\delta>0$ such that
for $0<\tau<\delta,$ $$ \frac12\lambda_2(S_P)< \hat{\lambda}_2(\tau)  < \cdots < \hat{\lambda}_{|B_P|}(\tau).$$ Moreover, by $\lim_{\tau\to0^+}\hat{\lambda}_1(\tau)=0,$ there exists $\delta_1$ such that for $0<\tau<\delta_1,$ $\hat{\lambda}_1(\tau)<\frac12\lambda_2(S_P).$ Hence, for $0<\tau<\min\{\delta,\delta_1\},$
\[
0=\hat{\lambda}_0(\tau) < \hat{\lambda}_1(\tau)  < \hat{\lambda}_2(\tau) < \cdots < \hat{\lambda}_{|B_P|}(\tau).
\] This proves that $\widehat{S}(\tau)$ has simple eigenvalues.


Pull back the weight $\widehat{w}$ to $T^{(1)}$ by replacing the effective edge of weight $\tau$ with the pendant path $P_{ub}$, choosing positive edge resistances, i.e., reciprocal of the weight, on $P_{ub}$ whose sum is $1/\tau$. This yields a weight $w^{(1)}$ for which $(T^{(1)},B^{(1)},w^{(1)})$ has simple Steklov spectrum.

Choose another boundary vertex $b' \in B \setminus B^{(1)}$ that is connected to $T^{(1)}$ by an interior path.
Let $T^{(2)}$ be the minimal subtree of $T$ containing $B^{(1)} \cup \{b'\}$.
Repeating the same argument as above — attaching a pendant edge (after contracting interior vertices) with sufficiently small weight — we can find a weight $w^{(2)}$ such that $(T^{(2)}, B^{(2)}, w^{(2)})$ has simple Steklov spectrum.
Continuing this process inductively, after adding all remaining boundary vertices one by one, we eventually obtain a weighted tree $(T, B, w_T)$ whose Steklov eigenvalues are all simple.

This completes the proof.
\end{proof}

\noindent{\textbf{Part 2 of Theorem \ref{thm:generic}}}
All  Steklov eigenfunctions of $(G,B)$ are nowhere zero on $B$ for generic weights.

\begin{proof}
Recall that $\chi(\lambda,w)$ denotes the characteristic polynomial of the Steklov matrix and  $\chi_i(\lambda,w)$ denotes the characteristic polynomial of the principal submatrix $S_{[B\setminus\{i\},\,B\setminus\{i\}]}$ for $i\in B$. Consider the resultant
\begin{align*}
\operatorname{Discr}_0:\mathbb{R}_{+}^{|E|}
&\to \mathbb{R}, \\
w
&\mapsto
\operatorname{Res}\left(\chi(\lambda,w),\prod_{i\in B}\chi_i(\lambda,w)\right).
\end{align*}
Note that the Steklov matrix is real symmetric, hence is Hermitian. By Proposition \ref{pro:zero},  $\operatorname{Discr}_0(w)=0$ if and only if $(G,B,w)$ admits a Steklov eigenfunction that vanishes at some boundary vertex.
As in the proof of simplicity, it therefore suffices to construct a positive weight function $w_T$ on a spanning tree $T$ for which every Steklov eigenfunction is nowhere vanishing on $B$.

Applying Corollary \ref{cor:path} to the weighted path $(P,B_P,w_P)$ in $(T,B,w_T)$, we may choose a positive weight $w_P$ so that
$(P,B_P,w_P)$ has simple Steklov spectrum. In addition, we may require the harmonic extension of every Steklov eigenfunction to be nowhere vanishing.
Let $b\in B\setminus B_P$ be a boundary vertex that is connected by an interior path in $T$  to some vertex in $B_P$.
Recall that $(T^{(1)},B^{(1)},w^{(1)})$ is the minimal weighted subtree of $T$ containing $B_P\cup\{b\}$, and $(\widehat{T},B^{(1)},\widehat{w})$
is obtained by replacing $P_{ub}$ in $T^{(1)}$ by a single edge $\{u,b\}$ with weight $\tau$.
For sufficiently small $\tau$, the Steklov eigenvalues of $(\widehat{T},B^{(1)},\widehat{w})$ can be ordered as in \eqref{eq:eigenorder} with $\hat{\lambda}_i:=\hat{\lambda}_i(\tau),$ $0\leq i\leq |B_P|.$

We first consider the Steklov eigenfunctions associated with $\hat{\lambda}_i$ for $i=2,\ldots, |B_P|$.
From \eqref{eq:S},  each entry of the Steklov matrix $\widehat{S}=\widehat{S}(\tau)$
associated with $(\widehat{T},B^{(1)},\widehat{w})$ is an analytic function of the parameter $\tau$. In particular, when $\tau=0$, we have $\widehat{S}(0)=\begin{pmatrix}
S_P & \mathbf{0}  \\[4pt]
\mathbf{0}^{\top} & 0
\end{pmatrix}$.
Let $(\lambda,\mathbf{x}_P)$ be an eigenpair of $S_P$ with $\lambda \neq 0$.
Then $\lambda$ is a simple eigenvalue of $\widehat{S}(0)$, and the corresponding eigenvector  is  $\begin{pmatrix}
\mathbf{x}_P\\[2pt] 0
\end{pmatrix}$.
By sensitivity theorem (Theorem \ref{thm:sensi}), there exists real analytic functions $\lambda(\tau)$ and $\begin{pmatrix}
\mathbf{x}_P(\tau)\\[2pt] \mathbf{x}_b(\tau)
\end{pmatrix}$ defined in a neighborhood of $\tau=0$, such that $\lambda(\tau)$ and $\begin{pmatrix}
\mathbf{x}_P(\tau)\\[2pt] \mathbf{x}_b(\tau)
\end{pmatrix}$ are respectively an eigenvalue and an associated eigenvector of $\widehat{S}(\tau)$.
Moreover, we have $\lambda(0)=\lambda$, $\mathbf{x}_P(0)=\mathbf{x}_P$ and $\mathbf{x}_b(0)=0$.

Since every entry of $\mathbf{x}_P$ is nonzero and
$\mathbf{x}_P(\tau)$ depends analytically on $\tau$,
there exists $\delta>0$ such that every entry of
$\mathbf{x}_P(\tau)$ is nonzero for all
$0<\tau<\delta$.

We next show that $\mathbf{x}_b(\tau)\neq 0$ for all
$0<\tau<\delta$. Suppose, to the contrary, that $\mathbf{x}_b(\tau)=0$ for some $0<\tau<\delta$.
Then $\mathbf{x}(\tau)=
\begin{pmatrix}
\mathbf{x}_P(\tau)\\[2pt]
0
\end{pmatrix}
$
is an eigenvector of $\widehat S(\tau)$ with eigenvalue $\lambda(\tau)$.
First assume that $u\in B_P$. By the eigenvalue equation and
\eqref{eq:S}, we have
\begin{align}
S_P\mathbf{x}_P(\tau)
+\tau\mathbf e_u\mathbf e_u^{\top}
\mathbf{x}_P(\tau)
&=
\lambda(\tau)\mathbf{x}_P(\tau),
\label{eq:eigf}\\
\tau\mathbf e_u^{\top}\mathbf{x}_P(\tau)
&=0.
\label{eq:eigg}
\end{align}
Since $\tau>0$, \eqref{eq:eigg} implies
$\mathbf e_u^{\top}\mathbf{x}_P(\tau)=0$. Substituting this into
\eqref{eq:eigf}, we obtain
\[
S_P\mathbf{x}_P(\tau)
=
\lambda(\tau)\mathbf{x}_P(\tau).
\]
Thus we get a Steklov eigenfunction $f$ of $(P,B_P,w_P)$ such that $f(u)=\mathbf e_u^{\top}\mathbf{x}_P(\tau)=0$, contradicting the choice of $w_P$.

Now assume that $u\in V(P) \setminus B_P$. By the eigenvalue equation and
\eqref{eq:S}, we similarly obtain
\[
S_P\mathbf{x}_P(\tau)
=
\lambda(\tau)\mathbf{x}_P(\tau),
\qquad
\mathbf p_u^{\top}\mathbf{x}_P(\tau)=0.
\]
Thus we get a Steklov eigenfunction $f$ of $(P,B_P,w_P)$ such that
\[
h_f(u)
=
\sum_{z\in B_P}\mathcal P(u,z)f(z)
=
\mathbf p_u^{\top}\mathbf{x}_P(\tau)
=
0,
\]
where $\mathcal P(u,z)$ is the discrete Poisson kernel, and the harmonic
extension $h_f$ of $f$ vanishes at the interior vertex $u$. This also contradicts the choice of $w_P$.
Therefore, $\mathbf{x}_b(\tau)\neq 0$ for all $0<\tau<\delta$. For all sufficiently small $\tau>0$, we conclude that the Steklov eigenfunctions of $(\widehat T,B^{(1)},\widehat w)$ corresponding to $\hat\lambda_i$, $i=2,\ldots,|B_P|$, are nowhere vanishing on $B^{(1)}$.

By the proof in Part 1 in Theorem~\ref{thm:generic}, we know that for sufficiently small $\tau>0,$ $\hat\lambda_1(\tau)$ is simple. Let $\mathbf{x}(\tau)$ be the  Steklov eigenvector corresponding to $\hat{\lambda}_1(\tau)$, satisfying $\|\mathbf{x}(\tau)\|=1$ and $\mathbf{x}_b(\tau)\geq 0.$ Note that the stability theorem is not directly applicable at $\tau=0$, since the zero eigenvalue of $\widehat S(0)$ is not simple. 

We claim that \[\lim_{\tau\to 0^+}\mathbf{x}(\tau)=
c
\begin{pmatrix}
-\mathbf{1}_{|B_P|}\\[2pt]
|B_P|
\end{pmatrix}, 
\] where $c=\frac{1}{\sqrt{|B_P|\,(|B_P|+1)}}.$
It suffices to prove that the set of limit points of
$\mathbf{x}(\tau)$ as $\tau\to0^+$ consists of a single element. Let $\mathbf{x}^*$ be any limit point  of $\mathbf{x}(\tau)$ along a subsequence of $\tau\to0^+.$ By the eigenequation for $\lambda_1(\tau)$ and passing to the limit $\tau\to0^+,$ up to a subsequence, we know that
$$\widehat S(0) \mathbf{x}^*=0.$$
One easily sees that
\[
\ker \widehat S(0)
= \operatorname{span}\left\{
\begin{pmatrix}
\mathbf{1}_{|B_P|}\\[2pt] 0
\end{pmatrix},
\begin{pmatrix}
\mathbf{0}_{|B_P|}\\[2pt] 1
\end{pmatrix}
\right\}.
\]
Moreover, since for small $ \tau,$ $\mathbf{x}(\tau)\perp\mathbf{1}_{|B^{(1)}|}$  and $\mathbf{x}_b\geq 0,$  we have
$$\mathbf{x}^*\perp\mathbf{1}_{|B^{(1)}|}\quad \mathrm{and}\quad \mathbf{x}^*_b\geq 0.$$ Hence, by $\|\mathbf{x}^*\|=1,$
\[
\mathbf{x}^*=
\frac{1}{\sqrt{|B_P|\,(|B_P|+1)}}\begin{pmatrix}
-\mathbf{1}_{|B_P|}\\[2pt]
|B_P|
\end{pmatrix}.
\]
This proves the claim.  Consequently, there exists $\delta>0$ such that, for all $0<\tau<\delta$, the vector
$\mathbf{x}(\tau)$ has no zero entries.

In order to repeat the inductive construction, we need that the harmonic extension of every Steklov eigenfunction of $(T^{(1)},B^{(1)},w^{(1)})$ is nowhere vanishing on the interior vertices.
We first establish this property on the contracted configuration
$(\widehat T,B^{(1)},\widehat w)$ and then show that it is preserved under
the pullback of the weight to $T^{(1)}$.

On the contracted tree $\widehat T$, this follows from analytic
perturbation. Indeed, the value of the harmonic extension at each
interior vertex is a real analytic function of $\tau$. At
$\tau=0$, these values coincide with those of the harmonic
extensions on the weighted path $(P,B_P,w_P)$, which are nonzero by the
choice of $w_P$. Hence, after possibly decreasing $\delta>0$, all these
values remain nonzero for $0<\tau<\delta$.

It remains to pull back  to $T^{(1)}$ by replacing the effective edge of weight $\tau$ with the pendant path
$P_{ub}:x_0=u,x_1,\ldots,x_k=b$. As in the proof of Corollary \ref{cor:path}, choose positive edge resistances
$r_j=\frac{1}{w^{(1)}(x_{j-1},x_j)}$ for $j=1,\ldots,k$, 
satisfying $\sum_{j=1}^k r_j=\frac1{\tau}$. 
Equivalently, choose the intermediate resistance ratios
\[
\theta_j=
\frac{\sum_{t=1}^j r_t}
{\sum_{t=1}^k r_t}=\tau\sum_{t=1}^j r_t,
\qquad
j=1,\ldots,k-1,
\]
so that, for every Steklov eigenfunction $f$ of
$(T^{(1)},B^{(1)},w^{(1)})$, the affine interpolation
\[
h_f(x_j)
=
(1-\theta_j)h_f(u)+\theta_j h_f(b)
\]
is nonzero for every $j=1,\ldots,k-1$, where $h_f$ denotes the harmonic
extension of $f$. Since there are only finitely many eigenfunctions up to a scalar and
finitely many interior vertices on $P_{ub}$, only finitely many choices
of each $\theta_j$ are excluded. Therefore the $\theta_j$ can be chosen
to avoid all of them.
This completes the construction of $w^{(1)}$, for which the harmonic extension of every Steklov eigenfunction is not vanishing on $V(T^{(1)})$.

We now repeat the same inductive procedure as in the proof of simplicity. At each step, we contract the pendant path to be added into a single pendant edge and choose its weight sufficiently small. By the preceding argument, this preserves both the simplicity of the Steklov spectrum and the nonvanishing property of the harmonic extensions. Expanding the contracted edge back to the original pendant path does not introduce any new zeros, so the nonvanishing property is preserved on the enlarged subtree. After adding all boundary vertices, we obtain a positive weight $w_T$ in $T$ such that the harmonic extension of every Steklov eigenfunction of $(T,B,w_T)$ is nowhere vanishing in $V(T)$. 

This completes the proof.

\end{proof}

\section*{Acknowledgments}
Lixiang Chen was partially supported by the National Natural Science Foundation of China (No. 12501485). 
Bobo Hua was partially supported by the National Natural Science Foundation of China (No. 12371056). 
Yongtang Shi was partially supported by the National Natural Science Foundation of China (No. 12431013) and the Fundamental and Interdisciplinary Disciplines Breakthrough Plan of the Ministry of Education of China (JYB2025XDXM207).


\bibliographystyle{amsplain}
\bibliography{ref}
\end{document}